\documentclass{amsproc}

\newtheorem{theorem}{Theorem}[section]
\newtheorem{lemma}[theorem]{Lemma}
\newtheorem{claim}[theorem]{Claim}

\numberwithin{equation}{section}

\newcommand{\blankbox}[2]{%
  \parbox{\columnwidth}{\centering
    \setlength{\fboxsep}{0pt}%
    \fbox{\raisebox{0pt}[#2]{\hspace{#1}}}%
  }%
}

\newcommand{\bd}{\noindent {\sc Proof}.\ \ }

\newcommand{\abs}[1]{\lvert#1\rvert}
\newcommand{\R}{\mathbb{R}}
\newcommand{\Z}{\mathbb{Z}}
\newcommand{\N}{\mathbb{N}}
\newcommand{\supp}{\mathop{\mathrm{supp}}\nolimits}
\newcommand{\Leb}{\mathop{\mathit{Leb}}\nolimits}

\begin{document}

\title[Minimal diffeo.'s of the circle with measurable fundamental domains]%
{Minimal $C^1$-diffeomorphisms of the circle which admit
measurable fundamental domains}

\author{Hiroki Kodama}
\address{Graduate School of Mathematical Sciences, The University of Tokyo, 
3-8-1 Komaba, Meguro-ku, Tokyo, 153-8914 Japan
}
\email{kodama@ms.u-tokyo.ac.jp}

\author{Shigenori Matsumoto}
\address{Department of Mathematics, College of
Science and Technology, Nihon University, 1-8-14 Kanda, Surugadai,
Chiyoda-ku, Tokyo, 101-8308 Japan
}
\email{matsumo@math.cst.nihon-u.ac.jp}

\thanks{
The first author is partially supported by the Japan Society for the Promotion of Science (JSPS) through its ``Funding Program for World-Leading Innovative R\&D on Science and Technology (FIRST Program).''
The second author is partially supported by Grant-in-Aid for
Scientific Research (C) No.\ 20540096.
}
\subjclass[2000]{Primary 37E15,
secondary 37C05, 37A40.}
\keywords{diffeomorphism, minimality, rotation number,
ergodicity}

\date{\today}

\maketitle

\begin{abstract}
We construct, for each irrational number $\alpha$, a minimal 
$C^1$-diffeomorphism of the circle with rotation number $\alpha$
which admits a measurable fundamental domain with respect to
the Lebesgue measure.
\end{abstract}

\section{Introduction}

The concept of ergodicity is important not only for measure preserving
dynamical systems but also for systems which admits a
natural quasi-invariant
measure. Given a probability space $(X,\mu)$ and a transformation $T$
of $X$, $\mu$ is said to be {\em quasi-invariant} if the push forward
$T_*\mu$ is equivalent to $\mu$. In this case $T$ is called {\em
ergodic}
with respect to $\mu$, if a $T$-invariant Borel subset in $X$ is either null
or conull.

A diffeomorphism of a differentiable manifold always leaves the
Riemannian volume (also called the Lebesgue measure) quasi-invariant,
and one can ask if a given diffeomorphism is ergodic with respect to
the Lebesgue measure (below {\em ergodic} for short) or not. 
Answering a question of A.~Denjoy \cite{D}, 
M.~Herman (Chapt.~VII, p.~86, \cite{H}) showed that 
a $C^1$-diffeomorphism of the circle with derivative of bounded variation
is ergodic provided 
its rotation number is irrational.
See also Chapt.~12.7, p.~419, \cite{KH}.
Contrarily Oliveira and da Rocha \cite{OR} constructed a minimal $C^1$
diffeomorphism
of the circle which is not ergodic.

At the opposite extreme of the ergodicity lies the concept of measurable 
fundamental domains. 
Given
a transformation $T$ of a standard probability space $(X,\mu)$ leaving
$\mu$ quasi-invariant, a Borel subset $C$ of $X$ is called a
{\em measurable fundamental domain} if $T^nC$ ($n\in\Z$) is
mutually disjoint and the union $\cup_{n\in\Z}T^nC$ is conull.
In this case any Borel function on $C$ can be extended to a
$T$-invariant measurable function on $X$, and an ergodic component of $T$
is just a single orbit. The purpose of this paper is to show the
following
theorem.

\begin{theorem} \label{main}
For any irrational number $\alpha$, there is a minimal 
$C^1$-diffeomorphism of the circle with rotation number $\alpha$
which admits a measurable fundamental domain with respect to
the Lebesgue measure. 
\end{theorem}

The rest of the paper is devoted to the proof of Theorem \ref{main}.

\section{A measurable fundamental domain for a Lipschitz homeomorphism}
\label{Lip}

For the sake of simplicity, we first consider 
the case of $\alpha=(\sqrt 5 - 1)/2$, 
the inverse of golden ratio.
We regard the circle $S^1$ as $\R/\Z$.
Suppose $R$ denotes the rotation by $\alpha$.
Define a Cantor set $C$ in the circle by
$$
C=\bigl\{ \sum_{k = 1}^\infty \frac{\varepsilon_k}{2^{3^k}}  \,\bigm|\,
     \text{ $\varepsilon_k = 0$ or $1$} \bigr\} \pmod{\Z}.
$$
Note that all numbers in $C$ are well approximable by rational numbers.

\begin{claim} \label{nointersection}
$R^nC \cap R^mC = \emptyset$ for any integers $n \neq m$.
\end{claim}

\begin{proof} 
Suppose $x \in R^nC \cap R^mC$, then
$x-n\alpha, x-m\alpha \in C$, therefore
$$(-n+m)\alpha \in C+(-C) := \bigl\{ \sum_{k \geq 3} \frac{\varepsilon'_k}{2^{3^k}}  \,\bigm|\,
     \text{ $\varepsilon'_k = 0$ or $\pm 1$} \bigr\} \pmod{\Z}.
$$
$(-n+m)\alpha$ is badly approximable, while $C+(-C)$ consists of 
 well approximable numbers,
which is a contradiction. 
\end{proof}
See Section \ref{Cantorset} for the detail and the general case.

\bigskip

Fix a probability measure $\mu_0$ on $C$ without atom 
such that $\supp(\mu_0)=C$. 
We also choose a sequence $(a_i)_{i\in\Z}$ of positive numbers satisfying
$\sum_{i\in\Z}a_i = 1$. Now we can define a 
probability measure $\mu$ on $S^1$ by
$$ \mu := \sum_{i\in\Z} a_i R^i_*\mu_0.$$
The Radon-Nikodym derivative $\frac{dR^{-1}_* \mu}{d\mu}$
is equal to $\frac{a_{i+1}}{a_i}$ on the set $R^iC$.
Now we assume that $\frac{a_{i+1}}{a_i} \in [\frac{1}{D},D]$
for some $D>1$, then it follows that
$\frac{dR^{-1}_* \mu}{d\mu} \in L^{\infty}(S^1,\mu)$.

We define a homeomorphism $h$ of $S^1$ by
$h(0)=0$ and $h(x) = y$ if and only if $\Leb [0,x] = \mu[0,y]$, where
$\Leb$ denotes the Lebesgue measure on $S^1$; or more briefly, 
$h_* \Leb = \mu$.
Finally define a homeomorphism $F$ of $S^1$ 
by $F:=h^{-1}\circ R \circ h$, then
$$
\frac{dF^{-1}_* \Leb}{d\Leb} = \frac{dR^{-1}_*\mu}{d\mu}\circ h 
\in L^\infty(S^1,\Leb),
$$
i.e.\ the map $F$ is a Lipschitz homeomorphism.
The set $C' = h^{-1} C$ is a measurable fundamental domain of $F$.

\section{Make it $C^1$}

\subsection{What shall we do?}

To prove the Theorem, we are trying to make
 the Radon-Nikodym derivative 
$g=\frac{dR^{-1}_*\mu}{d\mu}$
continuous on $S^1$.

Fix an arbitrary point $x_0 \in C$. For a positive integer $i$,

\begin{equation*}
\begin{split}
a_i &= (a_i/a_{i-1})\cdots(a_2/a_1)(a_1/a_0)a_0 \\
{}    &= g(R^{i-1}x_0) \cdots g(Rx_0)g(x_0)a_0,\\
a_{-i} &= (a_{-i+1}/a_{-i})^{-1}\cdots(a_{-1}/a_{-2})^{-1}(a_0/a_{-1})^{-1}a_0\\
{}       &= g(R^{-i}x_0)^{-1} \cdots g(R^{-2}x_0)^{-1}g(R^{-1}x_0)^{-1}a_0.
\end{split}
\end{equation*}

Set $\phi = \log g$ and 
define a map $\Phi\colon S^1\times\R \to S^1\times\R$ by
$\Phi(x,y) = (Rx,y+\phi(x))$. 
Simple calculation shows that $\Phi^n(x,y)=(R^n x,y+\phi^{(n)}(x))$ where
\begin{equation}\label{phim}
\begin{split}
\phi^{(m)}(x) &= \sum_{i=0}^{m-1} \phi(R^i x)\qquad \text{($m>0$),} \\
\phi^{(-m)}(x) &= -\sum_{i=1}^{m} \phi(R^{-i} x) \qquad \text{($m>0$),}\\
\phi^{(0)}(x) &= 0 .
\end{split}
\end{equation}

Therefore $a_i = \exp(\phi^{(i)}(x_0))a_0$. To satisfy
$\sum_{i\in\Z}a_i = 1$, it suffice to find $\phi$ so that
$\sum_{i\in\Z}\exp(\phi^{(i)}(x_0)) < \infty$.

\subsection{Construction step I}
Now we forget about $a_i$'s and the Cantor set $C$.
As a first step, we are going to construct a function $\phi\in C(S^1)$
satisfying
$ \sum_{i\in\Z} \exp(\phi^{(i)}(x_0)) < \infty$
for a single point $x_0$,
where $\phi^{(i)}$ are defined by (\ref{phim}).

We are going to define continuous functions $\phi_n\in C(S^1)$ ($n\in\N)$
in such a way that $\sum_{i=1}^\infty \lVert \phi_n \rVert_\infty < \infty$.
Then $\phi=\sum_{i=1}^\infty\phi_n$ converges uniformly, 
thus $\phi$ is also continuous.

Fix an integer $n\in\N$. Choose a sufficiently small neighbourhood $J$ 
of $x_0$
so that $R^{-2^n}J,\dots,R^{-1}J,J,RJ,\dots,R^{2^n-1}J$ are disjoint.
Consider a bump function $f$ on $J$ so that
$\supp f \subset J$, $f(x_0) = (3/4)^n$ and $0\leq f(x)<(3/4)^n$ 
on $J\setminus \{x_0\}$. Define $\phi_n \colon S^1 \to \R$ by
$$
\phi_n(x) =
\begin{cases}
-f(R^{-i} x) & x \in R^iJ,\, i=0,1,\dots,2^n-1\\
 f(R^{-i} x) & x \in R^iJ,\, i=-2^n,-2^n+1,\dots,-1\\
 0 & \text{otherwise.}
\end{cases}
$$

\begin{lemma}
$\phi_n^{(i)}(x_0)
\begin{cases}
= -\abs{i} (3/4)^n & \text{for $-2^n \leq i \leq 2^n$}\\
\leq 0 & \text{for any $i\in\Z$.}
\end{cases}
$
\end{lemma}

\begin{figure}[htbp]
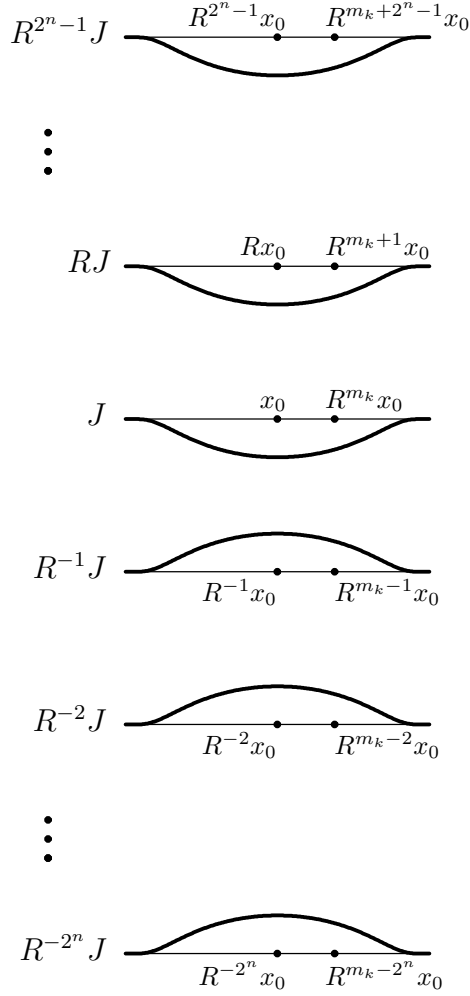

 \begin{center}
\unitlength 0.1in
%
  \vspace{2mm}
  \caption{The graph of $\phi_n$ and the orbit of $x_0$}
 \end{center}
\end{figure}

\begin{proof}
The equality for the first case is trivial.
Define an increasing sequence $(m_k)_{k\in\Z}$ by
$m_0=0$ and $\{m_k | k\in\Z\} = \{m\in\Z | R^m x_0 \in J \}$.
Since $R^{-2^n}J,\dots,R^{2^n-1}J$ are disjoint,
$m_{k+1}-m_k \geq 2^{n+1}$ for any $k\in\Z$.
Using this sequence, $R^m x_0 \in R^i J$ if and only if
$m=m_k+i$ for some $k$. Therefore, 
$$
\phi^{(i+1)}(x_0)
=
\begin{cases}
\phi^{(i)}(x_0) & m_{k-1}+2^n \leq i < m_k-2^n \\
\phi^{(i)}(x_0) + f(R^{m_k} x_0) & m_k-2^n \leq i < m_k  \\
\phi^{(i)}(x_0) - f(R^{m_k} x_0) & m_k \leq i < m_k+2^n \\
\end{cases}
$$
for some $k$. Induction for $\abs{k}$ shows that
$$
\phi^{(i)}(x_0)
=
\begin{cases}
- 2^n \cdot (3/4)^n & m_{k-1}+2^n \leq i \leq m_k-2^n \\
- 2^n \cdot (3/4)^n + (i-(m_k-2^n))f(R^{m_k} x_0) & m_k-2^n \leq i \leq m_k \\
- 2^n \cdot (3/4)^n + ((m_k+2^n)-i)f(R^{m_k} x_0) & m_k \leq i \leq m_k+2^n.\\
\end{cases}
$$
Since $f(R^{m_k} x_0) \leq (3/4)^n$, 
the inequality $\phi^{(i)}(x_0) \leq 0$ also holds.

\end{proof}

Therefore, if $2^n \leq \abs{i} < 2^{n+1}$,
$
\phi^{(i)}(x_0) 
\leq \phi_{n+1}^{(i)}(x_0) 
= -\abs{i} (3/4)^{n+1}
\leq - 2^n \cdot (3/4)^{n+1}
= - 3/4 \cdot (3/2)^n
$.
Finally, 
$\sum_{i\in\Z} \exp(\phi^{(i)}(x_0)) 
\leq 1 + \sum_{n=0}^\infty 2^{n+1}\exp(- 3/4 \cdot (3/2)^n)
=M<\infty$.

\subsection{Construction step II}

We will execute the same argument for the Cantor set $C$
instead of the single point $x_0$.
Since
$R^{-2^n}C,\dots,C,\dots,R^{2^n-1}C$ are disjoint compact sets,
there exists an $\varepsilon$-neighbourhood $N$ of $C$ such that
$R^{-2^n}N,\dots,N,$
$\dots,R^{2^n-1}N$ are disjoint.
Define a bump function $f$ so that
$\supp f \subset N$, $f(x) = (3/4)^n$ on $C$ and $0\leq f(x)<(3/4)^n$ 
on $N\setminus C$. Now we apply the same argument as in the previous subsection,to obtain the function $\phi \in C(S^1)$ such that 
$ \sum_{i\in\Z} \exp(\phi^{(i)}(x)) < M <\infty$ for any $x\in C$.

We define a finite measure $\tilde{\mu}$ on $S^1$ by
$$ \tilde{\mu} := \sum_{i\in\Z} (\exp \circ \phi^{(i)}\circ R^{-i}) R^i_*\mu_0.$$
Normalize $\tilde{\mu}$ to obtain a probability measure $\mu$, namely
$ \mu := \frac{\tilde{\mu}}{\int_{S^1}d\tilde{\mu}}$.

Define $h$ and $F$ as in the section \ref{Lip}, then
$$
\frac{dF^{-1}_* \Leb}{d\Leb} = \frac{dR^{-1}_*\mu}{d\mu}\circ h 
= g\circ h
$$
is a continuous function because $g(x)=\exp(\phi(x))$. 
We proved Theorem \ref{main} for the case $\alpha=(\sqrt 5 - 1)/2$.


\section{Construction of Cantor set for general $\alpha$}\label{Cantorset}

For general irrational number $\alpha$, 
it is enough to find a Cantor set $C$ that satisfies
Claim \ref{nointersection}.
For a real number $x$ and
a function $p \colon \N \to \N$,
define the approximation constant $c_p(x)$ by
$$ c_p(x) := \liminf_{q \to \infty} 
\left( p(q) \cdot \mathop{\rm dist}(x,\tfrac{1}{q}\Z) \right).$$
A real number $x$ is said to be $p$-approximable if $c_p(x)=0$.
Note that $x$ is well approximable if $x$ is $p$-approximable
for $p(q)=q^2$, so this is a generalization of well-approximability.

It is clear that $c_p(x)=0$ if $x$ is rational number.
On the other hand, for any irrational number $x$ we can find 
a function $p$ satisfying $c_p(x)>0$. Moreover, we will show the
following lemma.

\begin{lemma} \label{1}
For a given irrational number $\alpha$, we can find 
a function $p$ such that
$c_p(m \alpha)\geq 1$ for any nonzero integer $m$.
\end{lemma}

\begin{proof}
Since $c_p(-mx)=c_p(mx)$, 
it is enough to show the lemma for the case  $m \in \N$.
Let us start for any natural numbers $n$ and $q$ 
by taking a natural number $p_n(q)$ so that 
$p_n(q) \cdot \mathop{\rm dist}(n\alpha,\tfrac{1}{q}\Z) \geq 1$.
Then define a function $p$ by 
$$p(q) = \max_{1 \leq n \leq q} p_n(q).$$
By this construction $p(q) \geq p_m(q)$ for any $q\geq m$, therefore
\begin{equation*}
\begin{split}
c_p(m\alpha) 
&= \liminf_{q \to \infty} \left( p(q) \cdot \mathop{\rm dist}(m\alpha,\tfrac{1}{q}\Z) \right) \\
&\geq \liminf_{q \to \infty} \left( p_m(q) \cdot \mathop{\rm dist}(m\alpha,\tfrac{1}{q}\Z) \right) \\
&\geq 1.
\end{split}
\end{equation*}
\end{proof}
For this 
function $p$, 
we inductively take an increasing sequence $q_0,q_1,\dots$ of natural numbers 
satisfying the following conditions;
$q_0=1$, $q_n | q_{n+1}$, $q_n/q_{n+1} \leq 1/3 $ and $p(q_n) / q_{n+1} \leq 2^{-n}$.
Define a Cantor set $C$ by
$$
C := \bigl\{ \sum_{n=1}^{\infty}  \frac{\varepsilon_n}{q_n} \bigm| 
\text{$\varepsilon_n = 0$ or $1$} \bigr\}.
$$
This Cantor set $C$ consists of $p$-approximable numbers.
We can also show the following lemma.

\begin{lemma} \label{2}
For any $\beta \in C-C$, the approximation constant $c_p(\beta)$ is 
equal to $0$, where
$$
C-C = \bigl\{ \sum_{n=1}^{\infty}  \frac{\varepsilon'_n}{q_n} \bigm| 
\text{$\varepsilon'_n = 0$ or $\pm1$} \bigr\}.
$$
\end{lemma}

\begin{proof}
\begin{equation*}
\begin{split}
&p(q_i) \cdot \mathop{\rm dist}(\beta,\tfrac{1}{q_i}\Z)
\leq p(q_i) \left\lvert \sum_{n=i+1}^{\infty}  \frac{\varepsilon'_n}{q_n} \right\rvert 
\\ & \qquad
\leq p(q_i) \sum_{n=i+1}^{\infty} \frac{1}{q_n}
= \frac{p(q_i)}{q_{i+1}} \sum_{n=i+1}^{\infty} \frac{q_{i+1}}{q_n}
\leq \frac{1}{2^i} \sum_{k=0}^{\infty} \left( \frac{1}{3} \right)^k
= \frac{3}{2^{i+1}}.
\end{split}
\end{equation*}
Thus
\begin{equation*}
\begin{split}
c_p(\beta)
&= \liminf_{q \to \infty} \left( p(q) \cdot \mathop{\rm dist}(\beta,\tfrac{1}{q}\Z) \right)\\
&\leq \liminf_{i \to \infty} \left( p(q_i) \cdot \mathop{\rm dist}(\beta,\tfrac{1}{q_i}\Z) \right)\\
&\leq \liminf_{i \to \infty} \, \frac{3}{2^{i+1}}\\
&=0.
\end{split}
\end{equation*}
Therefore $c_p(\beta)=0$.
\end{proof}

Claim \ref{nointersection} follows from Lemma \ref{1} and Lemma \ref{2}, 
so we proved Theorem \ref{main} for the general case.


%
%

\end{document}